\documentclass[11pt,letterpaper]{article}
\usepackage[normalem]{ulem}
\usepackage{soul}
\usepackage{amsmath}
\usepackage{amsfonts}
\usepackage{amssymb}         %additional math symbols (see amsguide.tex)
\usepackage{amsopn}
\usepackage{theorem}          %additional LaTeX2e package (see pctex\latex2e)
\usepackage{latexsym}
\usepackage{graphicx}        %graphic and .eps files
\usepackage{mathrsfs}		%calligraphic fonts
\usepackage{psfrag}
\usepackage{algorithmic}
\usepackage{algorithm}
\usepackage{color}
\usepackage{verbatim}
\usepackage{url}

\newtheorem{result}{\ }[section]
\theoremstyle{changebreak}                % (see LATEX2E\THEOREM.DTX)
\newtheorem{theorem}[result]{Theorem}

\newtheorem{lemma}[result]{Lemma}

\newenvironment{proof}
 {{\sl Proof.}\hspace*{1 ex}}%
 {{\nopagebreak\hspace*{\fill}$\Box$\par\vspace{12pt}}}
\newcommand{\eps}{\varepsilon}

\definecolor{plgreen}{rgb}{0,0.5,0}
\newcommand{\LEO}[1]{{\color{black}#1}}
\newcommand{\KY}[1]{{\color{black}#1}}

\renewcommand{\Im}{\mbox{\sf Im}}
\newcommand{\Ims}{\mbox{\sf\scriptsize Im}}
\begin{document}

\thispagestyle{empty}
\begin{center} 

{\LARGE Random projections for trust region subproblems}
\par \bigskip
{\sc Ky Vu${}^1$, Pierre-Louis Poirion${}^2$, Claudia D'Ambrosio${}^3$, Leo Liberti${}^3$}
\par \bigskip
 {\small
   \begin{enumerate}
   \item {\it Chinese University of Hong Kong}
   \item {\it Huawei Research Center, Paris}
   \item {\it CNRS LIX, \'Ecole Polytechnique, F-91128 Palaiseau, France} \\ Email:\url{{dambrosio,liberti}@lix.polytechnique.fr}
   \end{enumerate}
 }
\par \medskip \today
\end{center}
\par \bigskip

% insert abstract
\begin{abstract}
  The trust region method is an algorithm traditionally used in the field of derivative free optimization. The method works by iteratively constructing surrogate models (often linear or quadratic functions) to approximate the true objective function inside some neighborhood of a current iterate. The neighborhood is called ``trust region'' in the sense that the model is trusted to be good enough inside the neighborhood. Updated points are found by solving the corresponding trust region subproblems. In this paper, we describe an application of random projections to solving trust region subproblems approximately. 
\end{abstract}

% insert paper
\section{Introduction}
Derivative free optimization (DFO) (see \cite{conn2009, Kramer2011}) is a field of optimization including techniques for solving optimization problems in the absence of derivatives. This often occurs in the presence of {\it black-box functions}, i.e.~functions for which no compact or computable representation is available. It may also occur for computable functions having unwieldy representations (e.g.~a worst-case exponential-time or simply inefficient evaluation algorithm). A DFO problem in general form is defined as follows:
\[ \min \; \{ f(x) \; | \; x \in \mathcal{D}\},\]
where $\mathcal{D}$ is a subset of $ \mathbb{R}^n$ and $f(\cdot)$ is a continuous function such that no derivative information about it is available \cite{Conn1997}.

As an example, consider a simulator $f(p,y)$ where $p$ are parameters and $y$ are state variables indexed by timesteps $t$ (so $y=(y^t\;|\;t\le h)$), including some given boundary conditions $y^0$. The output of the simulator is a vector at each timestep $t\le h$. We denote this output in function of parameters and state variables as $f^t(p,y)$. We also suppose that we collect some actual measurements from the process being simulated, say $\hat{f}^t$ for some $t\in H$, where $H$ is a proper index subset of $\{1,\ldots,h\}$. We would like to choose $p$ such that the behaviour of the simulator is as close as possible to the observed points $\hat{f}^t$ for $t\in H$. The resulting optimization problem is:
\[ \min \left\{ \bigoplus\limits_{t\in H} \|\hat{f}^t - f^t(p,y)\| \;\bigg|\; p \in \mathcal{P}\land y\in \mathcal{Y} \right\}, \]
where $\mathcal{P}$ and $\mathcal{Y}$ are appropriate domains for $p,y$ respectively, $\|\cdot\|$ is a given norm, and $\oplus$ is either $\sum$ or $\max$.

This is known to be a black-box optimization problem whenever $f(\cdot)$ is given as an {\it oracle}, i.e.~the value of $f$ can be obtained given inputs $p,y$, but no other estimate can be obtained directly. In particular, the derivatives of $f(\cdot)$ w.r.t.~$p$ and $y$ are assumed to be impossible to obtain. Note that the lack of derivatives essentially implies that it is hard to define and therefore compute local optima. This obviously makes it even harder to compute global optima. This is why most methods in DFO focuses on finding local optima.

Trust Region (TR) methods are considered among the most suitable methods for solving DFO problems \cite{Conn1996,conn2009,Marazzi2002}. TR methods involve the construction of surrogate models to approximate the true function (locally) in ``small'' subsets $D\subset\mathcal{D}$, and rely on those models to search for optimal solutions within $D$. Such subsets are called {\it trust regions} because the surrogate models are ``trusted'' to be ``good enough'' limited to $D$. TRs are often chosen to be closed balls $B(c,r)$ (with $c$ center and $r$ radius) with respect to some norms. There are several ways to obtain new data points, but the most common way is to find them as minima of the current model over the current trust region. Formally, one solves a sequence of so-called {\it TR subproblems}:
\begin{eqnarray*} 
 \min \, \{m(x) \; | \;  x  \in B(c, r) \cap \mathcal{D} \},
\end{eqnarray*}
where $m(\cdot)$ is a surrogate model of the true objective function $f(\cdot)$, which will then be evaluated at the solution of the TR subproblems. Depending on the discrepancy between the model and the true objective, the balls $B(c,r)$ and the models $m(\cdot)$ are updated: the TRs can change radius and/or center. The iterative application of this idea yields a TR method (we leave the explanation of TR methods somewhat vague since it is not our main focus; see \cite{conn2009} for more information).

In this paper we assume that the TR subproblem is defined by a linear or quadratic model $m(\cdot)$ and that $\mathcal{D}$ is a full-dimensional polyhedron defined by a set of linear inequality constraints. After appropriate scaling, the TR subproblem can be written as:
\begin{equation}  
  \min  \{x^{\top} Q x + c^{\top} x \; | \;  A x \le  b, \; \|x\| \le 1 \}, \label{TR:main-prob}
\end{equation}
where $Q \in \mathbb{R}^{n \times n}$ ($Q$ is not assumed to be positive semidefinite), $A \in \mathbb{R}^{m \times n}$, $b \in \mathbb{R}^{m}$. Here, $\| \cdot \|$ refers to Euclidean norm and such notation will be used consistently throughout the paper. This problem has been studied extensively \cite{Bienstock,Jeyakumar,Salahi2016,Burer2015}. Specifically, its complexity status varies between $\mathbf{P}$ and $\mathbf{NP}$-hard, meaning it is $\mathbf{NP}$-hard in general \cite{murty}, but many polynomially solvable sub-cases have been discovered. For example, without linear constraints $Ax \le b$  the TR subproblem can be solved in polynomial time \cite{Ye:1992}; in particular, it can be rewritten as a Semidefinite Programming (SDP) problem. When adding one linear constraint or two parallel constraints, it still remains polynomial-time solvable \cite{Ye03,sturm2003cones}; see the introduction to \cite{Bienstock} for a more comprehensive picture.

\subsection{The issue\dots}
In practice, the theoretical complexity of the TR subproblem Eq.~\eqref{TR:main-prob} is a moot point, since the time taken by each function evaluation $f(p,x)$ is normally expected to exceed the time taken to solve Eq.~\eqref{TR:main-prob}. Since the solution sought is not necessarily global, any local Nonlinear Programming (NLP) solver \cite{snopt,ipopt} can be deployed on Eq.~\eqref{TR:main-prob} to identify a locally optimal solution. The issue we address in this paper arises in cases when the sheer number of variables in Eq.~\eqref{TR:main-prob} prevents even local NLP solvers from converging to some local optima in acceptable solution times. The usefulness of TR methods is severely hampered if the TR subproblem solution phase represents a non-negligible fraction of the total solution time. As a method for addressing the issue we propose to solve the TR subproblem {\it approximately} in exchange for speed. We justify our choice because the general lack of optimality guarantees in DFO and black-box function optimization makes it impossible to evaluate the loss one would incur in trading off local optimality guarantees in the TR subproblem for an approximate solution. As an approximation method, we consider {\it random projections}.

Random projections are simple but powerful tools for dimension reduction \cite{Woodruff, vu2015, pilanci2014, jll_ctw16}. They are often constructed as random matrices sampled from some given distribution classes. The simplest examples are matrices sampled componentwise from independently identically distributed (i.i.d.) random variables with Gaussian $\mathcal{N}(0,1)$, uniform $[-1,1]$ or Rademacher $\pm 1$  distributions. Despite their simplicity, random projections are competitive with more popular dimensional reduction methods such as Principal Component Analysis (PCA) / Multi-Dimensional Scaling (MDS) \cite{jolliffe_10}, Isomap \cite{tenenbaum_00} and more. One of the most important features of a random projection is that it approximately preserves the norm of any given vector with high probability. In particular, let $P \in \mathbb{R}^{d \times n}$ be a random projection (such as those introduced above), then for any $x  \in \mathbb{R}^{n}$ and $\varepsilon \in (0,1)$, we have
\begin{equation} \label{RP}
\mbox{\sf Prob} \bigg[ (1 - \varepsilon) \|x\|^2 \le \|Px\|^2 \le (1 + \varepsilon) \|x\|^2\bigg] \ge 1-2e^{-\mathcal{C}\varepsilon^2 d},
\end{equation}
where $\mathcal{C}$ is a {\it universal constant} (in fact a more precise statement should be existentially quantified by ``there exists a constant $\mathcal{C}$ such that\dots'').

Perhaps the most famous application of random projections is the so-called \emph{Johnson-Lindenstrauss lemma} \cite{jllemma}. It states that for any $\varepsilon \in (0,1)$ and for any finite set $X \subseteq \mathbb{R}^n$, there is a mapping $F:\mathbb{R}^n \to \mathbb{R}^d$, in which $d = O(\frac{\log |X|}{\varepsilon^2})$, such that
\[\forall x, y \in X \qquad (1 - \varepsilon) \|x - y\|^2 \le \|F(x) - F(y)\|^2 \le (1 + \varepsilon) \|x - y\|^2 .\]
Such a mapping $F$ can be found as a realization of the random projection $P$ above; and the existence of the correct mapping is shown (by the probabilistic method) using the union bound. Moreover, the probability of sampling a correct mapping is also very high, i.e. in practice there is often no need to re-sample $P$.

\subsection{\dots and how we address it}
The object of this paper is the applicability of random projections to TR subproblems Eq.~\eqref{TR:main-prob}. Let $P \in \mathbb{R}^{d \times n}$ be a random projection with i.i.d.~$\mathcal{N}(0,1)$ entries. We want to ``project" each vector $x \in \mathbb{R}^n$ to a lower-dimension vector $Px \in \mathbb{R}^d$ and study the following \emph{projected problem}
\begin{eqnarray*} 
 \min \, \{x^{\top} (P^\top P Q P^\top P) x + c^{\top} P^\top P x \; | \;  A P^\top P x \le  b, \; \|Px\| \le 1 \}.
\end{eqnarray*}
By setting $u = Px, \; \bar{c} = Pc, \; \bar{A} = A P^\top$, we can rewrite it as 
\begin{eqnarray}  
 \min_{u \in \;\Ims(P)} \, \{u^{\top} (P Q P^\top) u + \bar{c}^{\top} u \; | \;  \bar{A} u \le  b, \; \|u\| \le 1 \},
\end{eqnarray}
where $\Im(P)$ is the image space generated by $P$. Intuitively, since $P$ is a projection from a (supposedly very high dimensional) space to a lower dimensional space, it is very likely to be a surjective mapping. Therefore, we assume it is safe to remove the constraint  $u \in \;\Im(P)$ and study the smaller dimensional problem:
\begin{eqnarray}  \label{TR:proj-prob}
\min_{u \in \mathbb{R}^d} \, \{u^{\top} (P Q P^\top) u + \bar{c}^{\top} u \; | \;  \bar{A} u \le  b, \; \|u\| \le 1 \},
\end{eqnarray}
where $u$ ranges in $\mathbb{R}^d$. As we will show later, Eq.~\eqref{TR:proj-prob}  yields a good approximate solution of the TR subproblem with very high probability.
%This result is interesting, especially for DFO and black-box problems. In these cases, it is unwise to spend too much time on solving TR subproblems and we are often happy with approximate solutions. Moreover, since the surrogate models might not even fit the true objective function, we are more or less tolerative to the small probability of failure. 

\section{Random projections for linear and quadratic models}
In this section, we will explain the motivations for the study of the projected problem (\ref{TR:proj-prob}).

We start with the following simple lemma, which says that linear and quadratic models can be approximated well using random projections.

\subsection{Approximation results}
\begin{lemma} \label{jll-approx}
Let $P: \mathbb{R}^n \to \mathbb{R}^d$ be a random projection satisfying Eq.~\eqref{RP} and let $0 < \varepsilon < 1$. Then there is a universal constant $\mathcal{C}_0$ such that
\begin{itemize}
\item[(i)] For any $x, y \in \mathbb{R}^n$:
\[\KY{\langle x, y\rangle - \varepsilon \|x\|\,\|y\| \le \langle Px, Py \rangle \le \langle x, y\rangle + \varepsilon \|x\|\,\|y\|}\]
with probability at least $1 - 4e^{-\mathcal{C}_0\varepsilon^2 d}$.
\item[(ii)] For any $x\in \mathbb{R}^n$ and $A \in  \mathbb{R}^{m \times n}$  whose rows are unit vectors:
$$ \KY{
	Ax - \varepsilon \|x\|  \begin{bmatrix}1 \\ \ldots \\ 1\end{bmatrix}  \le 
AP^{\top}Px \le Ax + \varepsilon \|x\|  \begin{bmatrix}1 \\ \ldots \\ 1\end{bmatrix} 
}$$
 with probability at least $1 - 4me^{-\mathcal{C}_0\varepsilon^2 d}$.
\item[(iii)] For any two vectors $x, y\in \mathbb{R}^n$ and a square matrix $Q \in  \mathbb{R}^{n \times n}$, then with probability at least $1 - 8ke^{-\mathcal{C}_0\varepsilon^2 d}$, we have:
$$\KY{
	x^{\top} Q y - 3\varepsilon \|x\|\, \|y\|\, \|Q\|_* \le x^{\top}P^{\top}PQP^{\top}Py \le x^{\top} Q y  + 3\varepsilon \|x\|\, \|y\|\, \|Q\|_*},  $$
in which $\|Q\|_*$ is the nuclear norm of $Q$ and $k$ is the rank of $Q$.
\end{itemize}
\end{lemma}

\begin{proof} \mbox{ } \\
(i) Let $\mathcal{C}_0$ be the same universal constant (denoted by $\mathcal{C}$) in Eq.~\eqref{RP}. By the property in Eq.~\eqref{RP}, for any two vectors $u+v, \, u-v$ and using the union bound, we have
 \begin{align*}
|\langle Pu, Pv \rangle  - \langle u, v \rangle| & = \frac{1}{4} \big|\|P(u+v)\|^2 - \|P(u-v)\|^2 -  \|u+v\|^2 + \|u-v\|^2 \big| \\
 & \le  \frac{1}{4} \big|\|P(u+v)\|^2  -  \|u+v\|^2 \big| + \frac{1}{4}\big|\|P(u-v)\|^2 -  \|u-v\|^2\big| \\
 & \le  \frac{\varepsilon}{4} (\|u+v\|^2 + \|u-v\|^2) = \frac{\varepsilon}{2} (\|u\|^2 + \|v\|^2),
 \end{align*}
 with probability at least $1 - 4 e^{-\mathcal{C}_0\varepsilon^2d}$.  Apply this result for $u= \frac{x}{\|x\|}$ and $v = \frac{y}{\|y\|}$, we obtain the desired inequality.

(ii)  Let $A_1,\ldots, A_m$ be (unit) row vectors of $A$. Then 
$$AP^\top Px - Ax = \begin{pmatrix} A_1^\top  P^\top Px - A_1^\top x\\ \ldots \\ A_m^\top  P^\top Px - A_m^\top  x\end{pmatrix} = 
\begin{pmatrix} \langle PA_1, Px \rangle  - \langle A_1, x \rangle \\ \ldots \\ \langle PA_m, Px \rangle  - \langle A_m, x \rangle\end{pmatrix}.$$
The claim follows by applying Part (i) and the union bound.

(iii) Let $Q = U\Sigma V^{\top}$ be the Singular Value Decomposition (SVD) of $Q$. Here $U, V$ are $(n \times k)$-real matrices with orthogonal unit column vectors $u_1, \ldots, u_k$ and $v_1,\ldots,v_k$, respectively and $\Sigma = \mbox{\sf diag} (\sigma_1,\ldots,\sigma_k)$ is a diagonal real matrix with positive entries. Denote by $\textbf{1}_k = (1,\ldots, 1)^\top$ the $k$-dimensional column vector of all $1$ entries. Since
\begin{align*}
x^{\top}P^{\top}PQP^{\top}Py & = (U^\top P^\top P x)^\top \Sigma (V^\top P^\top P y)  \\
& = \big[U^\top  x + U^\top (P^\top P - \mathbb{I}_n) x\big]^\top \Sigma \big[V^\top  y + V^\top (P^\top P - \mathbb{I}_n) y \big] 
\end{align*}
then the two inequalities that
\begin{align*}
 (U^\top x - \varepsilon \|x\| \textbf{1}_k)^\top \; \Sigma \;( V^\top y  - \varepsilon \|y\| \textbf{1}_k) \le x^{\top}P^{\top}PQP^{\top}Py \le  (U^\top x + \varepsilon \|x\| \textbf{1}_k)^\top \; \Sigma \;( V^\top y  + \varepsilon \|y\| \textbf{1}_k)
\end{align*}
occurs with probability at least $1 - 8ke^{-\mathcal{C}\varepsilon^2 d}$ (by applying part (ii) and the union bound). Moreover
\begin{align*}
 (U^\top x - \varepsilon \|x\| \textbf{1}_k)^\top \; \Sigma \;( V^\top y  - \varepsilon \|y\| \textbf{1}_k)  
  = & \; \;
  x^{\top}Qy -   \varepsilon \|x\| (\textbf{1}_k^\top \Sigma V^\top y)  -  \varepsilon \|y\|  (x^\top U \Sigma \textbf{1}_k ) + \varepsilon^2 \|x\|\, \|y\| \sum_{i=1}^k \sigma_i\\
  = & \; \;
  x^{\top}Qy -   \varepsilon(\sigma_1, \ldots, \sigma_k) \big( \|x\| V^\top y  +  \|y\| U^\top x \big)  + \varepsilon^2 \|x\|\, \|y\| \sum_{i=1}^k \sigma_i,
\end{align*}
and
\begin{align*}
 (U^\top x + \varepsilon \|x\| \textbf{1}_k)^\top \; \Sigma \;( V^\top y  + \varepsilon \|y\| \textbf{1}_k)  = & \; \;
x^{\top}Qy +   \varepsilon(\sigma_1, \ldots, \sigma_k) \big( \|x\| V^\top y  +  \|y\| U^\top x \big)  + \varepsilon^2 \|x\|\, \|y\| \sum_{i=1}^k \sigma_i. 
\end{align*}
Therefore,
\begin{align*}
 | x^{\top}P^{\top}PQP^{\top}Py - x^{\top}Qy|  
\le &\;\; 
 \|x\|\,\|y\| \left(2 \varepsilon \sqrt{\sum_{i=1}^k \sigma^2_i} + \varepsilon^2 \sum_{i=1}^k \sigma_i\right) \\ 
\le &\;\; 
3 \varepsilon \|x\|\,\|y\|  \sum_{i=1}^k \sigma_i =  3 \varepsilon \|x\|\,\|y\|\,  \|Q\|_*
\end{align*}
with probability at least $1 - 8ke^{-\mathcal{C}\varepsilon^2 d}$. 
\end{proof}

It is known that singular values of random matrices often concentrate around their expectations. In the case when the random matrix is sampled from Gaussian ensembles, this phenomenon is well-understood due to many current research efforts. The following lemma, which is proved in \cite{Zhang2013}, uses this phenomenon to show that when $P \in \mathbb{R}^{d \times n}$ is a Gaussian random matrix (with the number of row significantly smaller than the number of columns), then $PP^{\top}$ is very close to the identity matrix. 
\begin{lemma}[\cite{Zhang2013}] \label{Zhang}
Let $P \in \mathbb{R}^{d \times n}$ be a random matrix in which each entry is an i.i.d $\mathcal{N}(0, \frac{1}{\sqrt{n}})$ random variable. Then for any $\delta > 0$ and $0 < \varepsilon < \frac{1}{2}$, with probability at least $1 - \delta$, we have:
$$\|PP^{\top} - I\|_2 \le \varepsilon$$
provided that
\begin{equation} \label{condition1}
n \ge \frac{(d+1) \log (\frac{2d}{\delta})} {\mathcal{C}_1 \varepsilon^2},
\end{equation}
where $\|\;.\, \|_2$ is the spectral norm of the matrix and $\mathcal{C}_1 > \frac{1}{4}$ is some universal constant. 
\end{lemma}
This lemma also tells us that when we go from low to high dimensions, with high probability we can ensure that the norms of all the points endure small distortions. Indeed, for any vector $u \in \mathbb{R}^d$, then 
$$\|P^\top u\|^2 - \|u\|^2 = \langle P^\top u,P^\top u\rangle -\langle u,u\rangle = \langle (PP^\top - I) u, u\rangle  \in \mbox{range}(- \varepsilon \|u\|^2, \varepsilon \|u\|^2), $$
due to the Cauchy-Schwarz inequality. Moreover, it implies that $\|P\|_2 \le (1+\eps)$ with probability at least $1 - \delta$.

Condition \eqref{condition1} is not difficult to satisfy, since $d$ is often very small compared to $n$. It suffices that $n$ should be large enough to dominate the effect of $\frac{1}{\varepsilon^2}$.

\subsection{Trust region subproblems with linear models}
We will first work with a simple case, i.e.~when the surrogate models used in TR methods are linear: 
\begin{eqnarray}  \label{TR:main-prob-linear}
 \min \, \{c^{\top} x \; | \;  A x \le  b, \; \|x\| \le 1, x \in \mathbb{R}^n \}.
\end{eqnarray}
We will establish a relationship between problem \eqref{TR:main-prob-linear} and its corresponding projected problem:
\begin{align}   \tag{$P^{-}_{\varepsilon}$}
 \min \, \{(Pc)^{\top} u \; | \;  AP^{\top} u \le  b, \; \|u\| \le 1-\varepsilon, u \in \mathbb{R}^d \}.
\end{align}

Note that, in the above problem, we shrink the unit ball by $\varepsilon$. We first obtain the following feasibility result:
\begin{theorem}\label{thm:feasibility1}
Let $P \in \mathbb{R}^{d \times n}$ be a random matrix in which each entry is an i.i.d $\mathcal{N}(0, \frac{1}{\sqrt{n}})$ random variable. Let $\delta \in (0,1)$. Assume further that 
$$n \ge \frac{(d+1) \log (\frac{2d}{\delta})} {\mathcal{C}_1 \varepsilon^2},$$
for some universal constant $\mathcal{C}_1 > \frac{1}{4}$. Then with probability at least $1- \delta$, for any feasible solution $u$ of the projected problem  ($P^{-}_{\varepsilon}$), $P^\top u$ is also feasible for the original problem (\ref{TR:main-prob-linear}). 
\end{theorem}
We remark the following universal property of Theorem \ref{thm:feasibility1}: with a fixed probability, feasibility holds for all (instead of a specific) vectors $u$.
%We also note that, given the ``approximate'' role that TR subproblems play within DFO, we can replace the constraint $\|u\| \le 1 - \varepsilon$ by $\|u\| \le 1$.

\begin{proof}
Let $\mathcal{C}_1$ be as in Lemma \ref{Zhang}. Let $u$ be any feasible solution for the projected problem  ($P^{-}_{\varepsilon}$) and take $\hat{x} = P^\top u$. Then we have $A\hat{x} = AP^\top u \le b$ and
$$\|\hat{x}\|^2 = \| P^\top u \|^2 = u^\top P^\top P u = u^\top u + u^\top (P^\top P - I) u \le (1+\varepsilon) \|u\|^2,$$
with probability at least $1 - \delta$ (by Lemma \ref{Zhang}). This implies that $\|\hat{x}\| \le  (1+\varepsilon/2) \|u\|$; and since $\|u\| \le 1 - \eps$, we have
$$\|\hat{x}\| \le  (1+\frac{\varepsilon}{2})  (1-\varepsilon) < 1,$$
with probability at least $1 - \delta$, which proves the theorem.
\end{proof}

In order to estimate the quality of the objective function value, we define another projected problem, which can be considered as a relaxation of $(P^{-}_{\varepsilon})$ (we enlarge the feasible set by $\varepsilon$):
\begin{align}   \tag{$P^{+}_{\varepsilon}$}
 \min \, \{(Pc)^{\top} u \; | \;  AP^{\top} u \le  b + \varepsilon, \; \|u\| \le 1 + \varepsilon, u \in \mathbb{R}^d \}.
\end{align}

Intuitively, these two projected problems $(P^{-}_{\varepsilon})$ and $(P^{+}_{\varepsilon})$ are very close to each other when $\varepsilon$ is small enough (under some additional assumptions, such as the ``fullness" of the original polyhedron). Moreover, to make our discussions meaningful, we need to assume that they are feasible (in fact, it is enough to assume the feasibility of $(P^{-}_{\varepsilon})$ only). 

Let $u_{\varepsilon}^{-}$ and $u_{\varepsilon}^{+}$ be optimal solutions for these two problems, respectively. Denote by $x_{\varepsilon}^{-} = P^\top u_\varepsilon^{-}$ and $x_{\varepsilon}^{+} = P^\top u_\varepsilon^{+}$.  
Let $x^*$ be an optimal solution for the original problem (\ref{TR:main-prob}). We will bound $c^\top x^*$ between $c^\top x_{\varepsilon}^{-}$ and  $c^\top x_{\varepsilon}^{+}$, the two values that are expected to be approximately close to each other. 

\begin{theorem}
Let $P \in \mathbb{R}^{d \times n}$ be a random matrix in which each entry is an i.i.d $\mathcal{N}(0, \frac{1}{\sqrt{n}})$ random variable. Let $\delta \in (0,1)$. Then there are universal constants  $\mathcal{C}_0 > 1$ and $\mathcal{C}_1 > \frac{1}{4}$ such that if the two following conditions 
\[
d \ge \frac{\log (m/\delta)}{\mathcal{C}_0 \eps^2},
\quad \mbox{and } \quad n \ge \frac{(d+1) \log (\frac{2d}{\delta})} {\mathcal{C}_1 \varepsilon^2}, 
\]
are satisfied, we will have:\\
(i) With probability at least $1 - \delta $,  the solution $x_{\varepsilon}^{-}$  is feasible for the original problem  (\ref{TR:main-prob}). \\
 (ii)   With probability at least $1 - \delta$, we have:
 $$ c^{\top}x_{\varepsilon}^{-} \ge c^{\top} x^* \ge c^{\top}x_{\varepsilon}^{+} - \varepsilon \|c\|. $$
\end{theorem}

\begin{proof}
Select $\mathcal{C}_0$ and $\mathcal{C}_1$ as in Lemma \ref{jll-approx} and Lemma \ref{Zhang}. Note that the second condition is the requirement for the Lemma \ref{Zhang} to hold, and the first condition is equivalent to 
$ m \,e ^{-\mathcal{C}_0 \varepsilon^2 d} \le \delta$.  Therefore, we can choose the universal constant $\mathcal{C}_0$
such that $1 - (4m+6) \,e^{-\mathcal{C}_0\varepsilon^2 d} \ge 1 - \delta$.\\
(i) From the previous theorem, with probability at least $1 - \delta $, for any feasible point  $u$ of the projected problem  ($P^{-}_{\varepsilon}$), $P^\top u$ is also feasible for the original problem (\ref{TR:main-prob-linear}). Therefore, it must hold also for $x_{\varepsilon}^{-}$.

(ii) From part (i), with probability at least $1- \delta$, $x_{\varepsilon}^{-}$  is feasible for the original problem  (\ref{TR:main-prob}). Therefore, we have
\begin{align*}
c^{\top}x_{\varepsilon}^{-}  \ge c^{\top} x^* ,
\end{align*}
with probability at least $1- \delta$.

Moreover, due to Lemma \ref{jll-approx}, with probability at least $1 - 4e^{-\mathcal{C}_0\varepsilon^2 d}$, we have
\begin{align*}
 c^{\top} x^* \ge c^{\top} P^{\top} P x^* - \varepsilon \|c\| \,\|x^*\| \ge c^{\top} P^{\top} P x^* - \varepsilon \|c\|,
\end{align*}
(the last inequality follows from $\|x^*\| \le 1$).  On the other hand, let $\hat{u}= Px^*$, due to Lemma \ref{jll-approx}, we have
$$ AP^\top \hat{u} = AP^\top P x^*  \le  Ax^*+  \varepsilon \|x^*\|  \begin{bmatrix}1 \\ \cdots \\ 1\end{bmatrix} \le Ax^* +  \varepsilon \begin{bmatrix}1 \\ \cdots \\ 1\end{bmatrix} \le b + \eps,$$
 with probability at least $1 - 4me^{-\mathcal{C}_0\varepsilon^2 d}$, and
 $$\|\hat{u}\| = \|Px^*\| \le (1+ \eps) \|x^*\| \le (1+ \eps),$$
with probability at least $1 - 2e^{-\mathcal{C}_0\varepsilon^2 d}$, by Property \eqref{RP}.
Therefore, $\hat{u}$ is a feasible solution for the problem ($P^{+}_{\varepsilon}$) with probability at least $1 - (4m+2)e^{-\mathcal{C}_0\varepsilon^2 d}$. 

Since $u^+_\eps$ is the optimal solution for the problem ($P^{+}_{\varepsilon}$), it follows that
\begin{align*}
 c^{\top} x^* \ge  c^{\top} P^{\top} P x^* - \varepsilon \|c\| = c^{\top} P^{\top} \hat{u} - \varepsilon \|c\| \ge c^{\top} P^{\top} u_{\varepsilon}^{+} - \varepsilon \|c\| = c^{\top} x_{\varepsilon}^{+} - \varepsilon \|c\|,
\end{align*}
with probability at least $1 - (4m+6)e^{-\mathcal{C}_0\varepsilon^2 d}$, which is at least $1- \delta$ for our chosen universal constant $\mathcal{C}_0$.
\end{proof}

We have showed that $c^\top x^*$ is bounded between  $c^{\top}x_{\varepsilon}^{-}$ and $c^{\top}x_{\varepsilon}^{+} $. Now we will estimate the gap between these two bounds.  First of all, as stated at the outset, we assume that the feasible set 
$$S^* = \{ x \in \mathbb{R}^n \; | \; A x \le  b, \; \|x\| \le 1\}$$
is full-dimensional. This is a natural assumption since otherwise the polyhedron  $\{ x \in \mathbb{R}^n \; | \; A x \le  b\}$ is almost surely not full-dimensional either.

We associate with each set $S$ a positive number $\mbox{\sc full}(S) > 0$, which is considered as a fullness measure of $S$ and is defined as the maximum radius of any closed ball contained in $S$. Now, from the our assumption, we have $\mbox{\sc full} (S^*) = r^* > 0$, where $r^\ast$ is the radius of the greatest ball inscribed in $S^\ast$ (see Fig.~\ref{f:fullness}). 

\begin{figure}[h]
	\centering
	\includegraphics[width=10cm]{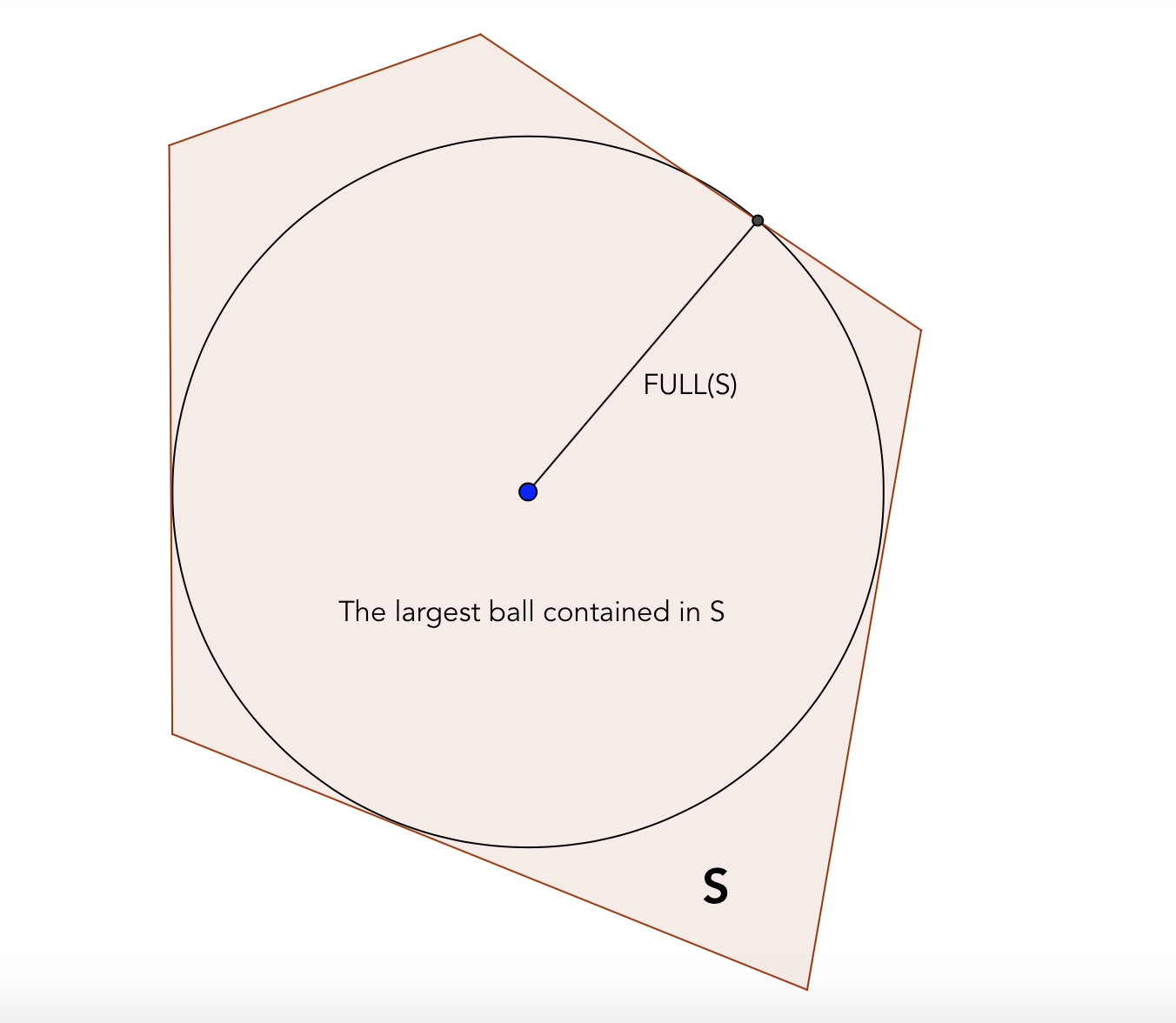}
	\caption{Fullness of a set.}
        \label{f:fullness}
\end{figure}

The following lemma characterizes the fullness of $S^+_\eps$ with respect to $r^*$, where
$$S_{\varepsilon}^{+}= \{ u \in \mathbb{R}^d \; | \; AP^{\top} u \le  b + \varepsilon, \; \|u\| \le 1 + \varepsilon\},$$
that is, the feasible set of the problem ($P_{\varepsilon}^{+}$).  
\begin{lemma}\label{fullness-lem}
	Let $S^*$ be full-dimensional with $\mbox{\sc full} (S^*)= r^*$. Then  
	with probability at least $1 - 3\delta$, $S^+_\eps$ is also full-dimensional with the fullness measure:
	$${\mbox{\sc full} }(S^+_\eps) \ge (1 - \eps) r^*.$$
\end{lemma}

In the proof of this lemma, we will extensively use the fact that, for any row vector $a \in \mathbb{R}^n$
$$\sup_{\|u\| \le r} a^\top u = r\|a\|,$$
which is actually the equality condition in the  Cauchy-Schwartz inequality. 

\begin{proof}
For any $i \in \{1,\ldots,n\},$ let $A_i$ denotes the $i$th row of $A$. Let $B(x_0, r^*)$ be a closed ball contained in $S^*$. Then for any $x \in \mathbb{R}^n$ with $\|x\| \le r^*$, we have $A(x_0 + x) = Ax_0 + Ax \le b,$ which implies that for any $i \in \{1,\ldots,n\}$,
\begin{equation}\label{eq:1}
b_i \ge (Ax_0)_i + \sup_{\|x\| \le r^*} A_ix = (Ax_0)_i + r^* \|A_i\|= (Ax_0)_i + r^*,
\end{equation}
hence $$b \ge Ax_0 + r^* \quad \mbox{or equivalently, } \quad Ax_0 \le b - r^*  . $$

By Lemma \ref{jll-approx}, with probability at least $1 - \delta$, we have
$$AP^\top P x_0 \le Ax_0 + \eps \le b -  r^* + \eps .$$
Let $u \in \mathbb{R}^n$ with $\|u\| \le (1-\eps)r^*$, then for any $i \in \{1,\ldots,n\}$, by the above inequality we have:
\begin{align*}
(AP^\top(Px_0+u))_i & = (AP^\top Px_0)_i + (AP^\top u)_i \\ & \le b_i+\eps -r^* + (AP^\top)_iu \\ 
&= b_i+\eps -r^* +A_iP^\top u \end{align*}
where $(AP^\top)_i$ denotes the $i$th row of $AP^\top$.
Since it holds for all such vector $u$, hence, by Cauchy-Schwartz inequality,
$$(AP^\top(Px_0+u))_i\le  b_i+\eps -r^* + (1-\eps)r^*\|A_iP^\top\| .$$
Using Eq. \eqref{RP} and the union bound, we can see that with probability at least $1-2me^{-\mathcal{C}\varepsilon^2 d}\ge 1-\delta$, we have: for all $i \in \{1,\ldots,m\}$, $\|A_iP^\top\|\le (1+\eps)\|A_i\|=(1+\eps)$. Hence
$$ AP^\top(Px_0+u)\le  b_i+\eps -r^* + (1-\eps)r^*(1+\eps) \le b+\eps $$
with probability at least $1-2\delta$. 

In other words, with probability at least $1-2\delta$, the closed ball $B^*$ centered at $Px_0$ with radius $(1-\eps)r^*$ is contained in $\{u\;|\; AP^\top u \le b + \eps\}$. 

Moreover, since $B(x_0, r^*)$ is contained in $S^*$, which is the subset of the unit ball, then $\|x_0\| \le 1 - r^*$.

With probability at least $1 - \delta$, for all vectors $u$ in $B(Px_0, (1-\eps)r^*)$, we have
$$\|u\|\le \|Px_0\| +  (1-\eps)r^* \le (1+\eps) \|x_0\| +  (1-\eps)r^* \le (1 + \eps)(1-r^*) + (1-\eps)r^* \le 1 + \eps.$$
Therefore, by the definition of $S^+_\eps$ we have 
$$B\big(Px_0, (1-\eps)r^*\big) \subseteq S^+_\eps,$$ which implies that the fullness of $S^+_\eps$ is at least $(1 - \eps) r^*$, with probability at least $1 - 3\delta$.
\end{proof}

Now we will estimate the gap between the two objective functions of the problems ($P^+_\varepsilon$) and ($P^-_\varepsilon$) using the fullness measure. The theorem states that as long as the fullness of the original polyhedron is large enough, the gap between them is small.
\begin{theorem}\label{thm:sandwich1}
	With probability at least $1 - 2\delta$, we have
	$$c^\top x^+_\eps \le c^{\top} x^-_\eps \le c^{\top}  x_{\varepsilon}^{+} +  \frac{18\eps}{\mbox{\sc full} (S^*) }\|c\|.$$ 
\end{theorem}
\begin{proof}

Let $B(u_0, r_0)$ a closed ball with maximum radius that is contained in $S^+_\eps$.

In order to establish the relation between $u^+_\eps$ and  $u^-_\eps$, our idea is to move $u^+_\eps$ closer to $u_0$ (defined in the above lemma), so that the new point is contained in $S^-_\eps$. Therefore, its objective value will be at least that of $u^-_\eps$, but quite close to the objective value of $u^+_\eps$.

\begin{figure}[h]
	\centering
	\includegraphics[width=8cm]{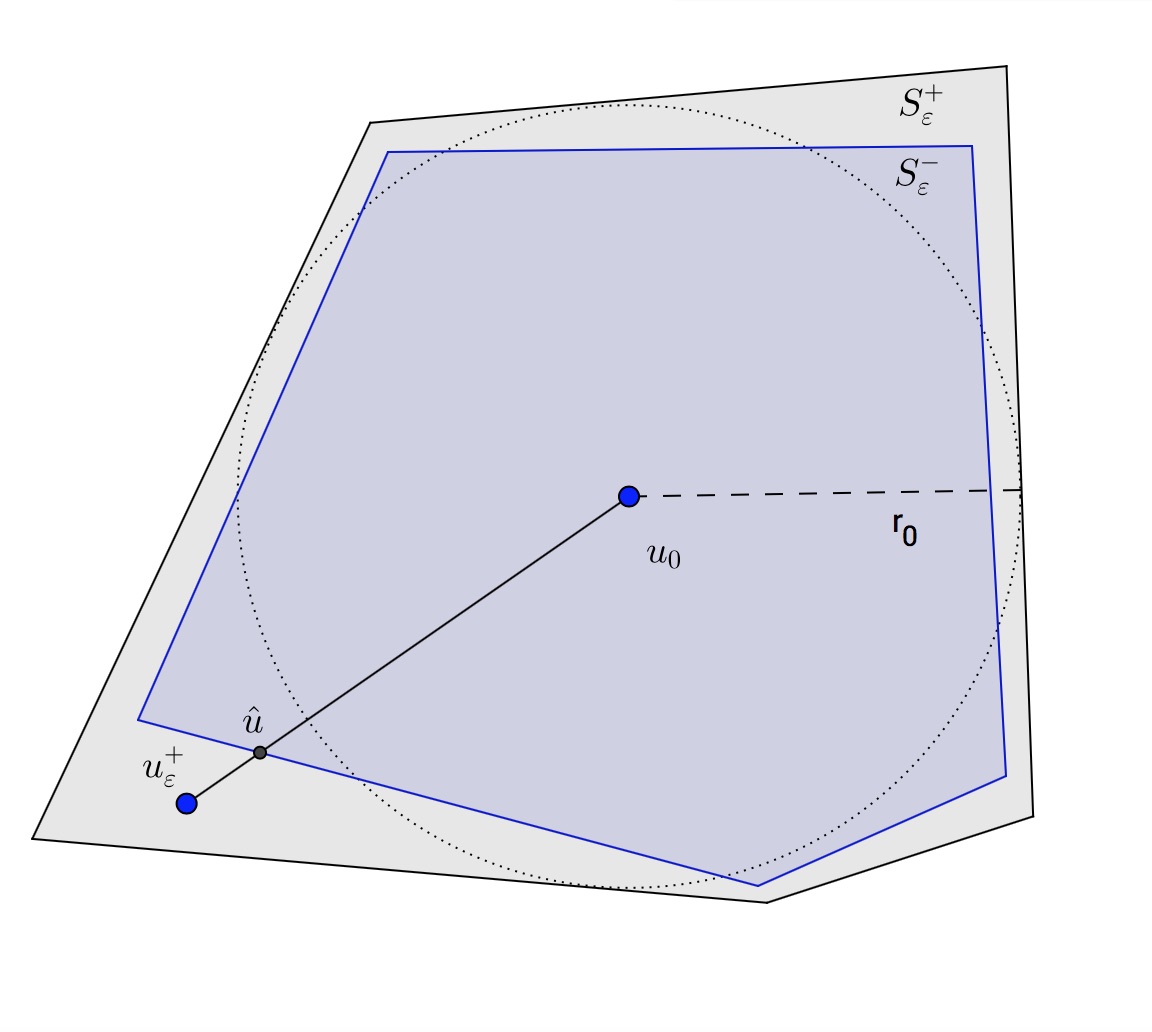}
	\caption{Idea of the proof}
\end{figure}

We define $\hat{u} : = (1-\lambda) u_{\varepsilon}^{+} + \lambda u_0$ for some $\lambda \in (0,1)$ specified later. We want to find $\lambda$ such that $\hat{u}$ is feasible for $P_{\varepsilon}^{-}$ while its corresponding objective value is not so different from $c^{\top}x_{\varepsilon}^{+} $. 

Since for all $\|u\| \le r_0$:
$$AP^{\top} (u_0 + u) = AP^{\top} u_0 + AP^{\top} u \le  b + \varepsilon,$$
then, similarly to Eq. \eqref{eq:1}, 
$$ AP^{\top} u_0 \le  b  + \varepsilon - r_0 \begin{pmatrix}
\|A_1P^\top\| \\
\vdots \\
\|A_mP^\top\|
\end{pmatrix} .$$
Therefore, we have, with probability at least $1-\delta$,
$$ AP^{\top} u_0 \le  b  + \varepsilon - r_0 (1-\eps) \begin{pmatrix}
\|A_1\| \\
\vdots \\
\|A_m\|
\end{pmatrix}=  b  + \varepsilon - r_0 (1-\eps).$$
Hence
$$AP^{\top} \hat{u} =  (1-\lambda) AP^{\top}u_{\varepsilon}^{+} + \lambda AP^{\top}u_0 \le b  + \varepsilon - \lambda r_0 (1-\eps) \le b  + \varepsilon - \frac{1}{2}\lambda r_0 ,$$
as we can assume w.l.o.g.~that $\eps \le \frac{1}{2}$. Hence, $AP^{\top} \hat{u} \le b$ if we choose $\varepsilon \le \lambda \frac{r_0}{2} $.
Moreover, let  $u\in B(u_0, r_0)$ such that $u$ is collinear with $0$ and $u_0$, and such that with $\|u\|=r_0$; we then have 
$$\|u_0\| \le  \|u_0 + u\| - r_0 \le 1 + \varepsilon - r_0,$$
so that
$$\|\hat{u}\| \le  (1-\lambda) \|u_{\varepsilon}^{+} \|+ \lambda \|u_0\| \le  (1-\lambda) (1 + \eps) + \lambda (1 + \varepsilon - r_0) = 1 + \eps - \lambda r_0,$$
which is less than or equal to $1-\eps$ for
$$\varepsilon \le \frac{1}{2} \lambda r_0.$$
 We now can choose $\lambda = 2\frac{\eps}{r_0 }$. With this choice, $\hat{u}$ is a feasible point for the problem $P^-_\eps$. Therefore,  we have
$$c^{\top} P^{\top} u^-_\eps \le c^{\top} P^{\top} \hat{u} = c^{\top} P^{\top} u_{\varepsilon}^{+} + \lambda c^{\top} P^{\top} (u_0 - u_{\varepsilon}^{+} ) \le 
c^{\top} P^{\top} u_{\varepsilon}^{+} +  \frac{4 (1+\eps)\eps}{r_0 }\|Pc\|.$$
By the above Lemma, we know that $r_0 \ge (1- \eps) r^*$, therefore, we have:
$$c^{\top} P^{\top} u^-_\eps \le c^{\top} P^{\top} \hat{u} \le 
c^{\top} P^{\top} u_{\varepsilon}^{+} +  \frac{4 (1+\eps)\eps}{r_0  }\|Pc\| \le  c^{\top} P^{\top} u_{\varepsilon}^{+} +  \frac{4(1+\eps)^2\eps}{(1- \eps) r^*  }\|c\|,$$
with probability at least $1 - \delta$.

 The claim of the theorem now follows, by noticing that, when $0 \le \eps \le \frac{1}{2}$, we can simplify:
$$\frac{2(1+\eps)^2}{(1- \eps)} \le \frac{2(1 + \frac{1}{2})^2}{(1- \frac{1}{2})} = 9.$$
\end{proof}

\subsection{Trust region subproblems with quadratic models}
In this subsection, we consider the case when the surrogate models using in TR methods are quadratic and defined as follows: 
\begin{eqnarray}  \label{TR:main-prob-quadratic}
\min \, \{x^\top Q x + c^\top x \; | \;  A x \le  b, \; \|x\| \le 1, x \in \mathbb{R}^n \}.
\end{eqnarray}
Similar to the previous section, we also study the relations between this and two other problems:

\begin{align}   \tag{$Q^{-}_{\varepsilon}$}
\min \, \{u^{\top} PQP^\top u + (Pc)^\top u\; | \;  AP^{\top} u \le  b, \; \|u\| \le 1-\varepsilon, u \in \mathbb{R}^d \} 
\end{align}
and 
\begin{align}   \tag{$Q^{+}_{\varepsilon}$}
\min \, \{u^{\top} PQP^\top u + (Pc)^\top u \; | \;  AP^{\top} u \le  b + \varepsilon, \; \|u\| \le 1 + \varepsilon, u \in \mathbb{R}^d \}. 
\end{align}

We will just state the following feasibility result, as the proof is very similar to that of Thm.~\ref{thm:feasibility1}.
\begin{theorem}
	\label{thm:feasibility2}
	Let $P \in \mathbb{R}^{d \times n}$ be a random matrix in which each entry is an i.i.d $\mathcal{N}(0, \frac{1}{\sqrt{n}})$ random variable. Let $\delta \in (0,1)$. Assume further that 
	$$n \ge \frac{(d+1) \log (\frac{2d}{\delta})} {\mathcal{C}_1 \varepsilon^2},$$
	for some universal constant $\mathcal{C}_1> \frac{1}{4}$. Then with probability at least $1- \delta$, for any feasible solution $u$ of the projected problem  ($Q^{-}_{\varepsilon}$), $P^\top u$ is also feasible for the original problem (\ref{TR:main-prob-quadratic}). 
\end{theorem}

Let $u_{\varepsilon}^{-}$ and $u_{\varepsilon}^{+}$ be optimal solutions for these two problems, respectively. Denote by $x_{\varepsilon}^{-} = P^\top u_\varepsilon^{-}$ and $x_{\varepsilon}^{+} = P^\top u_\varepsilon^{+}$.  
Let $x^*$ be an optimal solution for the original problem (\ref{TR:main-prob-quadratic}). We will bound $x^{*\top} Q x^*+ c^\top x^*$ between $x^{-\top}_\varepsilon Q x^-_\varepsilon + c^\top x^{-\top}_\varepsilon$ and  $x^{+\top}_\varepsilon Q x^+_\varepsilon + c^\top x^{+\top}_\varepsilon$, the two values that are expected to be approximately close to each other. 

\begin{theorem}
	Let $P \in \mathbb{R}^{d \times n}$ be a random matrix in which each entry is an i.i.d $\mathcal{N}(0, \frac{1}{\sqrt{n}})$ random variable. Let $\delta \in (0,1)$. 
	Let $x^*$ be an optimal solution for the original problems (\ref{TR:main-prob-quadratic}). 
	Then there are universal constants  $\mathcal{C}_0 > 1$ and $\mathcal{C}_1 > \frac{1}{4}$ such that if the two following conditions 
	\[
	d \ge \frac{\log (m/\delta)}{\mathcal{C}_0 \eps^2},
	\quad \mbox{and } \quad n \ge \frac{(d+1) \log (\frac{2d}{\delta})} {\mathcal{C}_1 \varepsilon^2}, 
	\]
	are satisfied, we will have:\\
	(i) With probability at least $1 - \delta $,  the solution $x_{\varepsilon}^{-}$  is feasible for the original problem  (\ref{TR:main-prob-quadratic}). \\
	(ii)   With probability at least $1 - \delta$, we have:
	$$ x_{\varepsilon}^{-\top} Q x_{\varepsilon}^{-} + c^\top x_{\varepsilon}^{-} \ge x^{*\top} Q x^{*\top}+c^\top  x^{*\top}  \ge x_{\varepsilon}^{+\top} Q x_{\varepsilon}^{+} + c^\top x_{\varepsilon}^{+}  - 3\varepsilon \|Q\|_* - \varepsilon \|c\|. $$
\end{theorem}
\begin{proof} The constants $\mathcal{C}_0$ and $\mathcal{C}_1$ are chosen in the same way as before. \\
	(i) From the previous theorem, with probability at least $1 - \delta $, for any feasible point  $u$ of the projected problem  ($P^{-}_{\varepsilon}$), $P^\top u$ is also feasible for the original problem (\ref{TR:main-prob-quadratic}). Therefore, it must hold also for $x_{\varepsilon}^{-}$.
	
	(ii) From part (i), with probability at least $1- \delta$, $x_{\varepsilon}^{-}$  is feasible for the original problem  (\ref{TR:main-prob-quadratic}). Therefore, we have
	\begin{align*}
	x_{\varepsilon}^{-\top} Q x_{\varepsilon}^{-} + c^\top x_{\varepsilon}^{-}\ge x^{*\top} Q x^{*} +c^\top x^{*},
	\end{align*}
	with probability at least $1- \delta$.
	
	Moreover, due to Lemma \ref{jll-approx}, with probability at least $1 - 8(k+1)e^{-\mathcal{C}_0\varepsilon^2 d}$, where $k$ is the rank of $Q$, we have
	\begin{align*}
	x^{*\top} Q x^{*} \ge x^{*\top} P^\top PQ P^\top P x^{*} - 3 \varepsilon \|x^*\|^2\, \|Q\|_*  \ge x^{*\top} P^\top PQ P^\top P x^{*} - 3 \varepsilon \|Q\|_*,
	\end{align*}
	and 
	\begin{align*}
	c^{\top} x^* \ge c^{\top} P^{\top} P x^* - \varepsilon \|c\|\, \|x^*\| \ge c^{\top} P^{\top} P x^* - \varepsilon \|c\|,
	\end{align*}
	since $\|x^*\| \le 1$. 
	Hence 
	\begin{align*}
	x^{*\top} Q x^{*} +	c^{\top} x^* \ge x^{*\top} P^\top PQ P^\top P x^{*} + c^{\top} P^{\top} P x^* - \varepsilon \|c\| - 3 \varepsilon \|Q\|_*,
	\end{align*}
	
	 On the other hand, let $\hat{u} = Px^*$, due to Lemma \ref{jll-approx}, we have
	$$ AP^\top \hat{u} = AP^\top P x^*  \le  Ax^*+  \varepsilon \|x^*\|  \begin{bmatrix}1 \\ \ldots \\ 1\end{bmatrix} \le Ax^* +  \varepsilon \begin{bmatrix}1 \\ \ldots \\ 1\end{bmatrix} \le b + \eps,$$
	with probability at least $1 - 4me^{-\mathcal{C}_0\varepsilon^2 d}$, and
	$$\|\hat{u}\| = \|Px^*\| \le (1+ \eps) \|x^*\| \le (1+ \eps),$$
	with probability at leats $1 - 2e^{-\mathcal{C}_0\varepsilon^2 d}$ (by Lemma \ref{jll-approx}).
	Therefore, $\hat{u}$ is a feasible solution for the problem ($P^{+}_{\varepsilon}$) with probability at least $1 - (4m+2)e^{-\mathcal{C}_0\varepsilon^2 d}$. Due to the optimality of $u^+_\eps$ for the problem ($P^{+}_{\varepsilon}$), it follows that
	\begin{align*}
	x^{*\top} Q x^{*} +	c^{\top} x^* &\ge x^{*\top} P^\top PQ P^\top P x^{*} + c^{\top} P^{\top} P x^* - \varepsilon \|c\| - 3 \varepsilon \|Q\|_*,\\ 
	&= \hat{u}^{\top} PQ P^\top \hat{u} + c^{\top} P^{\top} \hat{u}  - \varepsilon \|c\|- 3 \varepsilon \|Q\|_* \\ 
	&\ge u_\varepsilon^{+\top} PQ P^\top u_\varepsilon^{+} +(Pc)^\top  u_\varepsilon^{+\top} - \varepsilon \|c\| - 3 \varepsilon \|Q\|_* \\
	& = x_\eps^{+\top} Q x_\eps^{+} + c^\top x_\eps^{+\top}  - \varepsilon \|c\|- 3 \varepsilon \|Q\|_*,
	\end{align*}
	with probability at least $1 - (4m+6)e^{-\mathcal{C}_0\varepsilon^2 d}$, which is at least $1- \delta$ for the chosen universal constant $\mathcal{C}_0$.
	Moreover,
	\begin{align*}
	c^{\top} x^* \ge  c^{\top} P^{\top} P x^* - \varepsilon \|c\| = c^{\top} P^{\top} \hat{u} - \varepsilon \|c\| \ge c^{\top} P^{\top} u_{\varepsilon}^{+} - \varepsilon \|c\| = c^{\top} x_{\varepsilon}^{+} - \varepsilon \|c\|,
	\end{align*}
	Hence 
	$$x^{*\top} Q x^{*\top}+c^\top  x^{*\top}  \ge x_{\varepsilon}^{+\top} Q x_{\varepsilon}^{+} + c^\top x_{\varepsilon}^{+}  - 3\varepsilon \|Q\|_* - \varepsilon \|c\| $$
	which concludes the proof.
\end{proof}

The above result implies that the value of $x^{*\top} Q x^{*}+ c^\top x^* $ lies between $x_\eps^{-\top} Q x_\eps^{-} +c^\top x_\eps^{-} $ and $x_\eps^{+\top} Q x_\eps^{+} +c^\top x_\eps^{+} $. It remains to prove that these two values are not so far from each other.  For that, we also use the definition of fullness measure. We have the following result:
\begin{theorem}\label{thm:sandwich2}
	Let $0 < \varepsilon < 0.1$. Then with probability at least $1 - 4\delta$, we have
	$$ x_\eps^{+\top} Q x_\eps^{+}+ c^\top x_\eps^{+} \le x_\eps^{-} Q x_\eps^{-}+c^\top x_\eps^{-} < x_\eps^{+\top} Q x_\eps^{+} + c^\top x_\eps^{+} +  \frac{\eps}{\mbox{\sc full}(S^*)}(36+18\|c\|)  .$$ 
\end{theorem}
\begin{proof}
	Let $B(u_0, r_0)$ be a closed ball with maximum radius that is contained in $S^+_\eps$.
	
	We define $\hat{u} : = (1-\lambda) u_{\varepsilon}^{+} + \lambda u_0$ for some $\lambda \in (0,1)$ specified later. We want to find ``small" $\lambda$ such that $\hat{u}$ is feasible for $Q_{\varepsilon}^{-}$ while its corresponding objective value is still close to $x_\eps^{+\top} Q x_\eps^{+} + c^\top x_\eps^{+} $. 
	
	Similar to the proof of Th. \ref{thm:sandwich1}, when we choose 
 $\lambda = 2\frac{\eps}{r_0 }$, then $\hat{u}$ is feasible for the problem $Q^-_\eps$ with probability at least $1 - \delta$.
	
	Therefore,  $u^{-\top}_\eps PQP^\top u^{-}_\eps + (Pc)^\top  u^{-}_\eps$ is smaller than or equal to
	\begin{align*}
	& \mbox{ } \hat{u}^{\top} PQP^\top \hat{u} + (Pc)^\top \hat{u} \\ 
	&= \big(u^+_\eps + \lambda (u_0 - u^+_\eps ) \big)^{\top} PQP^\top \big(u^+_\eps + \lambda (u_0 - u^+_\eps ) \big)  + (Pc)^\top \hat{u}\\
	& = 
	u^{+\top}_\eps PQP^\top u^{+}_\eps + 
	\lambda u^{+\top}_\eps  PQP^\top \big(u_0 - u^+_\eps \big) + \lambda (u_0 - u^+_\eps )^{\top} PQP^\top u^+_\eps  \\  
	& \hspace*{1em} + \lambda^2 (u_0 - u^+_\eps )^{\top} PQP^\top (u_0 - u^+_\eps ) + (Pc)^\top \hat{u}.
	\end{align*}
	However, from Lemma \ref{jll-approx} and the Cauchy-Schwartz inequality, we have
	\begin{align*}
	u^{+\top}_\eps  PQP^\top \big(u_0 - u^+_\eps \big) & \le \|P^\top u^{+}_\eps\| \cdot \| Q \|_2 \cdot \|P^\top
	(u_0 - u^+_\eps \big) \| \\
	& \le (1 +\varepsilon)^2 \|u^{+}_\eps\| \cdot \| Q \|_2 \cdot \|
	(u_0 - u^+_\eps \big) \| \\
	& \le 2 (1+\varepsilon)^4 \; \|Q\|_2 \\
	& \quad \mbox{(Since $\|u^+_\varepsilon\|$ and $\|u^-_\varepsilon\| \le 1 + \varepsilon$)}
	\end{align*}
	Similar for other terms, then we have
	\begin{align*}
	\mbox{ } \hat{u}^{\top} PQP^\top \hat{u}  
	\le u^{+\top}_\eps PQP^\top u^{+}_\eps + 
	(4 \lambda + 4 \lambda^2) (1 + \varepsilon)^4 \; \|Q\|_2.
	\end{align*}
	Since $\eps < 0.1$, we have $(1+\varepsilon)^4 < 2$ and we can assume that $\lambda < 1$. Then we have
	\begin{align*}
	\hat{u}^{\top} PQP^\top \hat{u}  
	& < u^{+\top}_\eps PQP^\top u^{+}_\eps + 
	16 \lambda  \|Q\|_2 \\
	& = u^{+\top}_\eps PQP^\top u^{+}_\eps + \frac{32 \varepsilon}{r_0 }  \qquad \mbox{(since $\|Q\|_2 = 1$)}\\
	& \le u^{+\top}_\eps PQP^\top u^{+}_\eps + \frac{32 \varepsilon}{(1-\eps) \mbox{\sc full}(S^*) } 
	\qquad \mbox{(due to Lemma \ref{fullness-lem})}\\
	& < u^{+\top}_\eps PQP^\top u^{+}_\eps + \frac{36 \varepsilon}{\mbox{\sc full}(S^*) }  \qquad \mbox{(since $\varepsilon \le 0.1$)},
	\end{align*}
	with probability at least $1 - 2 \delta$.
	Furthermore, similarly to the proof of Th. \ref{thm:sandwich1}, we have 
	$$c^{\top} P^{\top} \hat{u} = (1-\lambda) (Pc)^\top u_{\varepsilon}^{+} + \lambda (Pc)^\top u_0  \le  c^{\top} P^{\top} u_{\varepsilon}^{+} +  \frac{4(1+\eps)^2\eps}{(1- \eps) r^*  }\|c\| \le  c^{\top} P^{\top} u_{\varepsilon}^{+} + \frac{18\eps}{\mbox{\sc full}(S^*)} \|c\|,$$ 
	with probability at least $1 - 2 \delta$.
\end{proof}

\section*{Conclusion}
In this paper, we have shown theoretically that random projection can be used in combination with trust region method to study high-dimensional derivative free optimization. We have proved that the solutions provided by solving low-dimensional projected version of trust region subproblems are good approximations of the true ones. We hope this provides a useful insight in the solution of very high dimensional derivative free problems.

\LEO{As a last remark, we note in passing that the results of this paper extend to the case of arbitrary (non-unit and non-zero) trust region radii, through scaling. This observation also suggests that the LP and QP trust region problem formulations analyzed in this paper can inner- and outer-approximate arbitrary bounded LPs and QPs.}

\bibliographystyle{plain}
\bibliography{jll_tr}

% end paper

\end{document}